\documentclass[11pt,a4paper,reqno]{amsart}
\usepackage{amsfonts}
\usepackage{amsthm}
\usepackage{amsmath}
\usepackage{amscd}
\usepackage[latin2]{inputenc}
\usepackage{t1enc}
\usepackage[mathscr]{eucal}
\usepackage{indentfirst}
\usepackage{graphicx}
\usepackage{graphics,color}
\usepackage{pict2e}
\usepackage{epic}
\usepackage{epstopdf}

\usepackage{hyperref}

\newtheorem{theorem}{Theorem}[section]
\newtheorem*{theorem*}{Theorem}
\newtheorem{proposition}[theorem]{Proposition}
\newtheorem{corollary}[theorem]{Corollary}
\newtheorem{lemma}[theorem]{Lemma}

\begin{document}

\title[Weighted geometric inequalities]{Weighted geometric inequalities for hypersurfaces in sub-static manifolds}

\author[F. Gir\~{a}o]{Frederico Gir\~{a}o}

\author[D. Rodrigues]{Diego Rodrigues}

\address{Universidade Federal do Cear\'{a}\\Departamento de Matem\'{a}tica\\Campus do Pici\\Av. Humberto Monte, s/n, Bloco 914, 60455-760\\Fortaleza/Ce\\Brazil} 
\email{fred@mat.ufc.br}

\address{Instituto Federal do Cear\'a\\
Av. Jos\'e de Freitas Queir\'os, 5000\\
63902-580\\
Quixad\'a/CE\\
Brazil}
\email{diego.sousa.ismart@gmail.com}

\thanks{Frederico Gir\~ao was partially supported by CNPq, grant number 306196/2016-6 and by FUNCAP/CNPq/PRONEX, grant number 00068.01.00/15. Diego Rodrigues was partially supported by a doctoral scholarship from CAPES}

\begin{abstract} 
We prove two weighted geometric inequalities that hold for strictly mean convex and star-shaped hypersurfaces in Euclidean space. The first one involves the weighted area and the area of the hypersurface and also the volume of the region enclosed by the hypersurface. The second one involves the total weighted mean curvature and the area of the hypersurface. Versions of the first inequality for the sphere and for the adS-Reissner-Nordstr\"om manifold are proven. We end with an example of a convex surface for which the ratio between the polar moment of inertia and the square of the area is less than that of the round sphere.
\end{abstract}

\maketitle

\section{Introduction}

Let $\Sigma$ be a closed orientable hypersurface embedded in $\mathbb{R}^n$ and assume that $\Sigma$ is strictly mean convex, that is, its mean curvature $H = (n-1)\sigma_1$ is positive. Let $\Omega$ be the region enclosed by $\Sigma$ and let $r$ be the distance to a fixed point $O$, which we refer to as the origin of $\mathbb{R}^n$. It is known that
\begin{equation} \label{kwong's inequality}
\int_\Sigma r\,d\Sigma \geq n \mathrm{Vol}(\Omega)    
\end{equation}
and
\begin{equation} \label{Kwong-Miao's inequality}
\int_\Sigma r^2\sigma_1 \, d\Sigma \geq n \mathrm{Vol}(\Omega).    
\end{equation}
Moreover, for any of the above inequalities, the equality holds if and only if $\Sigma$ is a sphere centered at the origin. Inequalities (\ref{kwong's inequality}) and (\ref{Kwong-Miao's inequality}) follow from the array of inequalities proved by Kwong in \cite{Kwong1} (Corollary 4.3).

Let $p \in \mathbb{R}^n$. A hypersurface $\Sigma$ in $\mathbb{R}^n$ is said to be {\em star-shaped with respect to $p$} if $\Sigma$ is the graph of some function defined on some geodesic sphere centered at $p$. We say that $\Sigma$ is {\em star-shaped} if there exists $p \in \mathbb{R}^n$ for which $\Sigma$ is star-shaped with respect to $p$. When the ambient is the hyperbolic space $\mathbb{H}^n$, these concepts are defined in a similar way.

When $\Sigma$ is star-shaped and strictly mean convex, inequality (\ref{Kwong-Miao's inequality}) was also proved by Kwong and Miao in \cite{kwong-miao}, using the inverse mean curvature flow (IMCF). 

Recall the isoperimetric inequality, which states that
$$
\omega_{n-1} \left( \frac{|\Sigma|}{\omega_{n-1}} \right)^{\frac{n}{n-1}} \geq n \mathrm{Vol}(\Omega),
$$
where $|\Sigma|$ denotes the area of $\Sigma$ and $\omega_{n-1}$ is the area of unit sphere $\mathbb{S}^{n-1}$.
Thus, it is natural to ask if (\ref{kwong's inequality}) and (\ref{Kwong-Miao's inequality}) can, respectively, be improved to
\begin{equation} \label{improved kwong}
    \int_\Sigma r\,d\Sigma \geq \omega_{n-1} \left( \frac{|\Sigma|}{\omega_{n-1}} \right)^{\frac{n}{n-1}}
\end{equation}
and
\begin{equation} \label{improved kwong-miao}
    \int_\Sigma r^2\sigma_1 \, d\Sigma \geq \omega_{n-1} \left( \frac{|\Sigma|}{\omega_{n-1}} \right)^{\frac{n}{n-1}}.
\end{equation}


We start by discussing inequality (\ref{improved kwong}), leaving inequality (\ref{improved kwong-miao}) for later.



Notice that, by Holder's inequality, if (\ref{improved kwong}) is holds, then
\begin{equation} \label{pmi inequality}
    \int_\Sigma r^2\,d\Sigma \geq \omega_{n-1}\left(\frac{|\Sigma|}{\omega_{n-1}}\right)^{\frac{n+1}{n-1}}
\end{equation}
also holds.

When $n=3$ we were able to construct a star-shaped and strictly convex hypersurface $\Sigma$ for which (\ref{pmi inequality}) doesn't hold
(see Section \ref{section_counterexample}). Obviously, for such surface, (\ref{improved kwong}) doesn't hold either.

Even though inequality (\ref{improved kwong}) isn't true (at least in dimension $n=3$), we were able to show the following result which, by the isoperimetric inequality, improves inequality (\ref{kwong's inequality}).

\begin{theorem}\label{ineq_rn} If $\Sigma$ is a smooth, star-shaped and strictly mean convex hypersurface in $\mathbb{R}^n$, then 
\begin{equation}\label{rn}
\int_\Sigma   r\,d\Sigma\geq \frac{n-1}{n}\omega_{n-1}\left(\frac{|\Sigma|}{\omega_{n-1}}\right)^{\frac{n}{n-1}}+\mbox{\rm Vol}(\Omega).
\end{equation}
Furthermore, the equality holds if and only if $\Sigma$ is a sphere centered at the origin.
\end{theorem}

The quantity on the left hand side of (\ref{pmi inequality}) is known as the {\em polar moment of inertia}. It is a very important quantity in Newtonian Physics. Our counterexample shows that the origin centered sphere is not a minimizer of the scale-invariant quantity
\begin{equation} \label{scale invariant pmi}
\left(\frac{|\Sigma|}{\omega_{n-1}}\right)^{-\frac{n+1}{n-1}}\int_\Sigma r^2\,d\Sigma
\end{equation}
among the family of strictly convex hypersurfaces, at least when $n=3$. 


A very interesting problem consists of finding the infimum of (\ref{scale invariant pmi}) over 
some family of (possibly nonsmooth) hypersurfaces. We now mention some papers that considered this problem (see each of the mentioned papers for details on the regularity of the family of hypersurfaces considered). For $n=2$, the problem was treated by Sachs in \cite{Sachs1,Sachs2}, where he proved, using geometric methods, that the infimum is achieved if and only if the curve is an origin centered equilateral triangle. An analytic proof of Sach's result was given by Hall in \cite{Hall}. When $n=3$, the problem was considered by Freitas, Laugesen and Liddell in \cite{FLL}, where they showed the existence of a minimizer over 
some suitable family of hypersurfaces.
They also conjectured that the infimum is attained when $\Sigma$ is some truncated tetrahedron.



As a consequence of the next result, which is a corollary of Theorem \ref{ineq_rn}, we have that among the family of star-shaped and strict mean convex hypersurfaces, the infimum of (\ref{scale invariant pmi}) is at least 
$$ 
\left( \frac{n-1}{n} \right)^2 \omega_{n-1}.
$$



\begin{corollary}\label{k=0}
If $\Sigma$ is star-shaped and strictly mean convex, then 
\begin{align} \label{inequality corollary}
\begin{aligned}
\int_\Sigma r^2\,d\Sigma \geq & \left(\frac{n-1}{n}\right)^2\omega_{n-1}\left(\frac{|\Sigma|}{\omega_{n-1}}\right)^{\frac{n+1}{n-1}} \\
&+\frac{2(n-1)}{n}\mbox{\rm Vol}(\Omega)\left(\frac{|\Sigma|}{\omega_{n-1}}\right)^{\frac{1}{n-1}}+\frac{\mbox{\rm Vol}(\Omega)^2}{|\Sigma|}.
\end{aligned}
\end{align}

with equality holding if and only if $\Sigma$ is a sphere centered at the origin.
\end{corollary}


Also in \cite{Kwong1}, analogs of (\ref{kwong's inequality}) and (\ref{Kwong-Miao's inequality}) were proved when the ambient space is taken to be another space form. For ambient spaces different from the Euclidean space, we will not deal with versions of (\ref{improved kwong-miao}), only with versions of (\ref{improved kwong}). Let us start by the case when the ambient is the sphere $\mathbb{S}^n$.

Recall that $(0,\pi)\times \mathbb{S}^{n-1}$ endowed with the metric
$$ dr^2 + \sin^2(r)h,$$
where $h$ is the metric of the unit sphere $\mathbb{S}^{n-1}$, gives a model for the round metric on $\mathbb{S}^n$. Here, $r$ is the geodesic distance to some fixed origin $O$.

Let $\Sigma$ be a smooth and strictly mean convex closed orientable hypersurface embedded in $\mathbb{S}^n$. It is proved in \cite{Kwong1} (Corollary 4.5) that if $\Sigma$ is contained in the open hemisphere centered at $O$, then
\begin{equation*}
    \int_\Sigma \sin r \, d\Sigma \geq n \int_\Omega \cos r \, d\Omega,
\end{equation*}
where $\Omega$ is the inner region enclosed by $\Sigma$.

In \cite{makowski-scheuer}, Makowski and Scheuer show that if $\Sigma$ is a strictly convex hypersurface embedded in $\mathbb{S}^n$, then the IMCF starting at $\Sigma$ converges, in finite time, to an equator $E_\Sigma$. This equator determines two hemispheres, with one of them containing $\Sigma$. We associate with each strictly convex $\Sigma \subset \mathbb{S}^n$ a point $x(\Sigma) \in \mathbb{S}^n$ in the following way: we let $x(\Sigma)$ be the center of the hemisphere, determined by $E_\Sigma$, that contains $\Sigma$ (looking at the hemisphere as a geodesic ball). We will refer to the point $x(\Sigma)$ as {\em the point associated to $\Sigma$ via the IMCF}. 

The following theorem is a version of Theorem \ref{ineq_rn} for hypersurfaces in $\mathbb{S}^n$.


\begin{theorem}\label{ineq_sn} Let $\Sigma$ be a smooth, strictly convex, closed orientable hypersurface in $\mathbb{S}^n$. Then
\begin{equation}\label{sn}
\int_\Sigma\sin r_x\geq\int_\Omega\cos r_x +\frac{n-1}{n}\omega_{n-1}\left(\frac{|\Sigma|}{\omega_{n-1}}\right)^{\frac{n}{n-1}},
\end{equation}
where $x$ is the point associated to $\Sigma$ via the IMCF and $r_x$ denotes the geodesic distance to $x$. The equality holds if and only if $\Sigma$ is a geodesic sphere centered at $x$.
\end{theorem}



Now, let $\mathbb{H}^n$ be the ambient space. Consider the following model for $\mathbb{H}^n$: the differentiable manifold $(0,\infty) \times \mathbb{S}^{n-1}$ endowed with the metric
\begin{equation*} 
dr^2 + \sinh^2(r)h,
\end{equation*}
where, as before, $h$ is the metric of the unit sphere $\mathbb{S}^{n-1}$.

It is proved in \cite{Kwong1} that if $\Sigma$ is a smooth and strictly mean convex closed orientable hypersurface embedded in $\mathbb{H}^n$, then
\begin{equation} \label{kwong's hyperbolic}
    \int_\Sigma \sinh(r) \, d\Sigma \geq n \int_\Omega \cosh(r) \, d\Omega,
\end{equation}
where, as before, $\Omega$ is the region enclosed by $\Sigma$. Moreover, the equality holds if and only if $\Sigma$ is a geodesic sphere centered at the origin.

Our analog of Theorem \ref{ineq_rn} will work not only for $\mathbb{H}^n$, but also for a family of ambient spaces, known as the adS-Reissner-Nordstr\"om family.

Let $(S^{n-1}, h_S)$ be a closed space form of sectional curvature $\epsilon \in \lbrace -1, 0, 1 \rbrace$. Let $m, q$ and $\kappa$, with $q < m,\kappa < \infty$ be such that the equation
$$
\epsilon + \kappa^2 s^2 - 2ms^{2-n} + q^2 s^{4-2n} = 0
$$
has positive real roots and let $s_0$ be the largest root of this equation. The adS-Reissner-Nordstr\"om manifold of mass $m$ and charge $q$ is the Riemannian manifold $(P,\gamma)$ defined as follows: $P = (s_0 , \infty) \times S$ and 
$$
\gamma:= \gamma_{m,q,\epsilon,\kappa} = \frac{1}{\epsilon + \kappa^2 s^2 - 2ms^{2-n} + q^2 s^{4-2n}} + s^2 h_S.
$$
The boundary $\partial P = \lbrace s_0 \rbrace \times S$ is referred to as the horizon of $(P,\gamma)$.


A hypersurface $\Sigma$ in $P$ is called star-shaped if $\Sigma$ is the graph of some function defined on the horizon.

It is known that, after a change of variable, the metric $\gamma$ can be written as 
\begin{equation*} 
dr^2 + \lambda^2(r)h_S,
\end{equation*} 
where $\lambda:[0,\infty) \to \mathbb{R}$ satisfies the ODE
\begin{equation} \label{ode}
\lambda'(r) = \sqrt{\epsilon + \kappa^2 \lambda^2 - 2m \lambda^{2-n} + q^2 \lambda^{4-2n}}
\end{equation}
(see \cite{wang-z}, Lemma 9).

Let $f = \lambda'(r)$. The function $f$ satisfies
$$
(\Delta_{\gamma}f)\gamma-\nabla_{\gamma}^2f+f{\rm Ric}_{\gamma}=(n-1)(n-2)q^2\lambda(r)^{4-2n}fh_S,
$$
where ${\rm Ric}_\gamma$, $\Delta_\gamma$ and $\nabla_\gamma^2$ are, respectively, the Ricci tensor, the Laplacian and the Hessian of the adS-Reissner-Nordstr\"om manifold $(P,\gamma)$.

Recall that a Riemannian manifold $(M,g)$ is called {\it sub-static} if 
\begin{equation} \label{sub-static}
(\Delta_gu)g-\nabla_g^2u+u{\rm Ric}_g\geq0
\end{equation}
for some positive function $u$.
Thus, the adS-Reissner-Nordstr\"om manifold is an example of a sub-static manifold. If the equality holds in (\ref{sub-static}), the manifold is said to be {\it static}. 
The Euclidean space and the sphere are examples of static manifolds.

Related to (\ref{kwong's hyperbolic}) we have the following analog of Theorem \ref{ineq_rn}:

\begin{theorem}\label{ineq_ads}
Let $\Sigma$ be a smooth, star-shaped and strictly mean convex hypersurface in the adS-Reissner-Nordstr\"om manifold $(P,\gamma)$ and let $\Omega$ denote the region bounded by $\Sigma$ and the horizon $\partial P$. Then
\begin{align}\label{ads}
\int_\Sigma\lambda\,d\Sigma\geq\int_\Omega\lambda'\,d\Omega+\frac{n-1}{n}\vartheta_{n-1}\left(\frac{|\Sigma|}{\vartheta_{n-1}}\right)^{\frac{n}{n-1}}+\frac{s_0}{n}|\partial P|,   
\end{align} 
where $\vartheta_{n-1} = |S|$. Moreover, equality holds if and only if $\Sigma$ is a slice, that is, $\Sigma = \lbrace s \rbrace \times S$ for some $s \in [s_0,\infty)$.

\end{theorem}

We remark that versions of Theorem \ref{ineq_ads} hold for the Schwarzschild space, the Kottler space, the adS-Schwarzschild space and the hyperbolic space. This can be seen by noticing that each of these spaces is the limit space of some subfamily of the adS-Reissner-Nordstr\"om family. For example, when the ambient is the hyperbolic space, we have the following version of Theorem \ref{ineq_ads}:


\begin{corollary}\label{ineq_hn}
Let $\Sigma$ be a smooth hypersurface in the hyperbolic space $\mathbb{H}^n$. If $\Sigma$ is star-shaped with respect to the origin and strictly mean convex, then
\begin{equation*}
\int_\Sigma \sinh (r) \,d\Sigma \geq \int_\Omega \cosh(r) \,d\Omega+\frac{n-1}{n}\omega_{n-1}\left(\frac{|\Sigma|}{\omega_{n-1}}\right)^{\frac{n}{n-1}}.    
\end{equation*} 
Morevover, equality holds if and only if $\Sigma$ is a geodesic sphere centered at the origin.
\end{corollary}



One can look at  (\ref{improved kwong-miao}) and (\ref{pmi inequality}) as weighted Alexandrov--Fenchel inequalities. We are now going to explore this point of view.

    
If $\Sigma$ is a convex hypersurface in $\mathbb{R}^n$, then the Alexandrov-Fenchel inequalities say that
     $$
\left( \frac{1}{\omega_{n-1}}\int_{\Sigma} \sigma_k(\lambda)\ \! d\Sigma \right)^{n-k} \geq \left( \frac{1}{\omega_{n-1}}\int_\Sigma \sigma_{k-1}(\lambda) \ \! d\Sigma \right)^{n-k-1},
$$
where $\sigma_k(\lambda)$, $1 \leq k \leq n-1$, is the normalized $k^{\mathrm{th}}$ elementary symmetric function of the principal curvature vector $\lambda = (\lambda_1, \ldots, \lambda_{n-1})$ of $\Sigma$. 
Moreover, the equality holds if and only if $\Sigma$ is a round sphere.

Guan and Li \cite{guan-li} showed that these inequalities (together with the rigidity statement) still hold if $\Sigma$ is only assumed to be star-shaped and $k$-convex (which means that $\sigma_i(\lambda) \geq 0$ for $i =  1, \ldots, k$).

It follows from the Alexandrov-Fenchel inequalities that
\begin{equation} \label{modified_AF}
\int_{\Sigma} \sigma_k(\lambda)\ \! d\Sigma \geq \omega_{n-1} \left( \frac{|\Sigma|}{\omega_{n-1}} \right)^{\frac{n-k-1}{n-1}},
\end{equation}
with the equality occurring if and only if $\Sigma$ is a round sphere. The inequalities (\ref{modified_AF}) were used by Ge, Wang and Wu in the proof of the Penrose inequality for asymptotically flat Euclidean graphs in the context of the Lovelock gravity \cite{GWW}.

In \cite{kwong-and-miao}, Kwong and Miao proved the following:
\begin{theorem*}
Let $k \in \lbrace 2, \ldots, n-1 \rbrace$. If $\Sigma \subset \mathbb{R}^n$ is such that $\sigma_k > 0$ on $\Sigma$, then
$$ \int_{\Sigma} r^2 \sigma_k \ \! d\Sigma \geq \int_{\Sigma} \sigma_{k-2} \ \! d\Sigma.$$
Moreover, the equality holds if and only if $\Sigma$ is a sphere centered at the origin.
\end{theorem*}
As a corollary, we get that for $k \in \lbrace 2, \ldots, n-1 \rbrace$,
\begin{equation} \label{Kwong-Miao}
\int_{\Sigma} r^2 \sigma_k \ \! d\Sigma \geq \omega_{n-1} \left( \frac{|\Sigma|}{\omega_{n-1}} \right)^{\frac{n-k+1}{n-1}},
\end{equation}
with the equality holding if and only if $\Sigma$ is an origin centered sphere.

Inequalities (\ref{Kwong-Miao}) can be seen as weighted versions of inequalities (\ref{modified_AF}). One can then ask if (\ref{Kwong-Miao}) remains true for $k \in \lbrace 0,1 \rbrace$. Notice that, the $k=1$ and $k=0$ cases of (\ref{Kwong-Miao}) are exactly the inequalities (\ref{improved kwong-miao}) and (\ref{pmi inequality}), respectively.

Using (\ref{Kwong-Miao's inequality}) and the IMCF, we were able to show (\ref{improved kwong-miao}) when $\Sigma$ is star-shaped and strictly mean convex.
\begin{theorem}\label{k=1}
If $\Sigma$ is star-shaped and strictly mean convex, then
\begin{equation} \label{improved kwong-miao again}
 \int_\Sigma r^2 \sigma_1 \ \! d\Sigma \geq \omega_{n-1} \left( \frac{|\Sigma|}{\omega_{n-1}} \right)^{\frac{n}{n-1}},
 \end{equation}
with the equality holding if and only if $\Sigma$ is a sphere centered at the origin.
\end{theorem}

\section{A monotone quantity along the IMCF on warped product manifolds}

Let $(N^{n-1},g_N)$ be a closed, orientable and connected Riemannian manifold. Let $a<b$ be positive real  numbers. We consider the product manifold $\overline{M}=N\times[a,b)$ equipped with the Riemannian metric
\begin{equation*} 
\overline{g}=dr^2+\lambda^2(r)g_N,
\end{equation*}
where $\lambda:[a,b)\to\mathbb{R}$ is a smooth function which is positive on $(a,b)$.
We allow the case in which $\lbrace a \rbrace \times N$ degenerates to a point, provided $\overline{M}$ is a smooth manifold and the restriction of $\overline{g}$ to $(a,b) \times N$ extends to a smooth Riemannian metric on $\overline{M}$.


Let $\Sigma$ be a closed, orientable and connected hypersurface embedded in $\overline{M}$. As observed in \cite{brendle}, $\overline{M}\setminus\Sigma$ has exactly two connected components, with exactly one of them contained in $N\times(a,b-\delta)$ for some $\delta>0$. We call this component the {\em inner region} and denote it by $\Omega$.
We either have $\partial\Omega=\Sigma$ or $\partial\Omega=\Sigma\cup(N\times\{a\})$. To simplify the notation, in the former case we let $\Gamma = \emptyset$, and in the later we let $\Gamma = \lbrace a \rbrace \times N$. Hence, no matter the case, we have $\partial \Omega = \Sigma \cup \Gamma$. We let $\nu$ be the outward-pointing unit normal to $\Sigma$. Also, whenever $\Gamma \not= \emptyset$, we let $\eta$ be the outward-pointing unit normal to $\Gamma$.

The following lemma is a generalization of (\ref{kwong's inequality}).

\begin{lemma}\label{key_lemma}
It holds
\begin{equation}\label{warp-ineq}
\int_\Sigma\lambda\,d\Sigma\geq n\int_\Omega \lambda'\,d\Omega+\lambda(a)|\Gamma|,
\end{equation}
with the equality occurring if and only if $\Sigma$ is a slice $\lbrace r \rbrace \times N$, for some $r \in [a,b)$.

 \begin{proof}
 We consider the vector field $Y=\lambda\partial_r$ in $\overline{M}$. Denoting by $\overline{\mbox{div}}$  the divergence with respect to $(\overline{M}, \overline{g})$, it is straightforward to verify that
 $$\overline{\mbox{div}}Y=n\lambda',$$
 where $\lambda'=dr(\lambda)$. Thus, by the divergence theorem, we have
 \begin{equation*}
 \begin{aligned}
n\int_{\Omega}\lambda'\,d\Omega&=\int_\Omega\overline{\mbox{div}}\ Y\,d\Omega\\
&=\int_\Sigma\langle Y,\nu\rangle\,d\Sigma+\int_{\Gamma}\langle Y,\eta\rangle\,d\Gamma\\
&=\int_\Sigma\lambda\langle\partial_r,\nu\rangle\,d\Sigma+\int_{\Gamma}\lambda\langle\partial_r,\eta\rangle\,d\Gamma\\
&=\int_\Sigma\lambda\langle\partial_r,\nu\rangle\,d\Sigma-\lambda(a)|\Gamma|\\
&\leq\int_\Sigma\lambda\,d\Sigma-\lambda(a)|\Gamma|,
\end{aligned}    
 \end{equation*}
with the equality holding if and only if $\langle \partial_r, \nu \rangle \equiv 1$ along $\Sigma$, which happens if and only if $\Sigma$ is a slice $\lbrace r \rbrace \times N$.

 \end{proof}
\end{lemma}

Let $\Sigma_0$ be a strictly mean convex hypersurface in $\overline{M}$ which is given by an embedding
$$
x_0: \Sigma \to \overline{M}.
$$

We consider an one-parameter family of embeddings
$$x:[0,T^*)\times \Sigma \to \overline{M},$$
which satisfy the flow equation
\begin{align}\label{imcf}
\begin{aligned}
    \frac{\partial x}{\partial t}=\frac{\nu}{H}\\ 
    x(0,\cdot)=x_0,
\end{aligned}
\end{align}
where, as before, $\nu$ is the outward-pointing unit normal vector to the hypersurface $\Sigma_t = x(t,\cdot)$ and $H = (n-1)\sigma_1$ is the mean curvature of $\Sigma_t$ with respect to this choice of unit normal. If no confusion arises, we denote the envolving hypersurface simply by $\Sigma$. 
The flow (\ref{imcf}) is the famous inverse mean curvature flow (IMCF).

\begin{proposition}
Under the IMCF, the following evolution equations hold:
\begin{itemize}
\item[(i)] The area element $d\Sigma$ evolves as
\begin{equation} \label{evolution area element}
    \frac{\partial}{\partial t} d\Sigma = d\Sigma;
\end{equation}
\item[] in particular, the area $|\Sigma|$ evolves as
\begin{equation} \label{evolution area}
    \frac{d}{dt} |\Sigma| = |\Sigma|.
\end{equation}
\item[(ii)] For any $u \in C^{\infty}(M)$, the quantity $$\int_\Omega u \, d\Omega$$ evolves as
\begin{equation} \label{deriv}
    \frac{d}{dt} \left( \int_{\Omega} u \, d\Omega \right) = \int_\Sigma \frac{u}{H} \, d\Sigma.
\end{equation}
\end{itemize}
\end{proposition}
\begin{proof}
Equations (\ref{evolution area element}) and (\ref{evolution area}) are well know (see, for example, \cite{huisken}). Equation (\ref{deriv}) follows from the co-area formula.
\end{proof}
\begin{proposition}\label{Key_result}
Let $\Sigma_0$ evolve by the IMCF. 
If $t \in [0,T^\ast)$ is such that $\lambda'(t) > 0$ and $\Sigma_{t}$ is strictly mean convex, then the quantity

\begin{align}
\label{monotone_quantity}
\begin{aligned}
    Q(t)= |\Sigma_t|^{-\frac{n}{n-1}} & \left(   \int_{\Sigma_t}\lambda\,d\Sigma-\int_{\Omega_t}\lambda'\,d\Omega  -\frac{\lambda(a)}{n}|\partial \overline{M}|\right)
\end{aligned}
\end{align}
satisfies $Q'(t) \leq 0$. Moreover, $Q'(t)=0$ if and only if $\Sigma_{t}$ is a slice $\lbrace r \rbrace \times N$.

\end{proposition}

\begin{proof}
Denote by $D$ the Levi-Civita connection of $(\overline{M},\overline{g})$. We have
\begin{equation}\label{auxliar}
\begin{aligned}
\left(\int_{\Sigma_t}\lambda\,d\Sigma\right)'&=\int_{\Sigma_t}\frac{\partial\lambda}{\partial t}\,d\Sigma+\int_{\Sigma_t}\lambda\,d\Sigma\\
&=\int_{\Sigma_t}\lambda'\left\langle Dr,\frac{\partial x}{\partial t}\right\rangle\,d\Sigma+\int_{\Sigma_t}\lambda\,d\Sigma\\
&=\int_{\Sigma_t}\frac{\lambda'}{H}\langle Dr,\nu\rangle\,d\Sigma+\int_{\Sigma_t}\lambda\,d\Sigma\\
&\leq\int_{\Sigma_t}\frac{\lambda'}{H}\,d\Sigma+\int_{\Sigma_t}\lambda\,d\Sigma\\
&=\left(\int_{\Omega_t}\lambda'\,d\Omega\right)'+\int_{\Sigma_t}\lambda\,d\Sigma,
\end{aligned}
\end{equation}
where he have used the Cauchy-Schwarz inequality
and (\ref{deriv}) with $u=\lambda'$. 
It follows from (\ref{warp-ineq})  and (\ref{auxliar}) that
\begin{equation*}
\begin{aligned}
\left(\int_{\Sigma_t}\lambda\,d\Sigma -\int_{\Omega_t}\lambda'\,d\Omega \ - \right. & \left.  \frac{\lambda(a)}{n}|\partial \overline{M}|\right)'  \\
&\leq\int_{\Sigma_t}\lambda\,d\Sigma\\
&=\frac{n}{n-1}\left(\int_\Sigma\lambda\,d\Sigma-\frac{1}{n}\int_\Sigma\lambda\,d\Sigma\right)\\
&\leq\frac{n}{n-1}\left(\int_\Sigma\lambda\,d\Sigma-\int_{\Omega}\lambda'\,d\Sigma-\frac{\lambda(a)}{n}|\partial \overline{M}|\right).
\end{aligned}
\end{equation*}
Also, from (\ref{evolution area}) we have $|\Sigma_t|'=|\Sigma_t|$. 
Thus, we conclude that $Q'(t)\leq0$. If $Q'(t) = 0$, then the equality holds in (\ref{warp-ineq}), which implies that $\Sigma_t$ is a slice $\lbrace r \rbrace \times N$. Also, one easily checks that if $\Sigma_t$ is a slice $\lbrace r \rbrace \times N$, then $Q'(t)=0$, since the equality holds in each of the inequalities.
\end{proof}

\section{Proof of the theorems}
Throughout this section, we let $\lbrace \Sigma_t \rbrace$ be the family of hypersurfaces obtained from the IMCF starting at $\Sigma$.
\subsection{The Euclidean space as the ambient space 
}

We will consider the following model for  $\mathbb{R}^n$: the differentiable manifold $(0,\infty)\times\mathbb{S}^{n-1}$ with the metric
$$\overline{g}=dr^2+r^2h.$$

The IMCF starting with a star-shaped and strictly mean convex hypersurface in $\mathbb{R}^n$ was treated by Gerhardt in \cite{Gerhardt} and by Urbas in \cite{Urbas}.
\begin{theorem}[\cite{Gerhardt} and \cite{Urbas}]\label{IMCF_in_rn}
 Let $\Sigma $ be a smooth, closed hypersurface in $ \mathbb{R}^{n}$ with positive mean curvature, given by
 a smooth embedding $x_0 : \mathbb{S}^{n-1} \rightarrow \mathbb{R}^{n} $.
 Suppose $ \Sigma$ is star-shaped with respect to a  point $ P $.
  Then the initial value problem
 \begin{equation*}
 \left\{
 \begin{aligned}
 \frac{\partial x}{\partial t}  = & \ \frac{1}{ H } \nu  \\
 x  (\cdot, 0) = & x_0 ( \cdot)
 \end{aligned}
 \right.
 \end{equation*}
 has a unique smooth solution $ x : \mathbb{S}^{n-1} \times [0, \infty) \rightarrow \mathbb{R}^{n} $, where
 $ \nu $ is the unit outer normal vector  to $\Sigma_t = x (\mathbb{S}^{n-1}, t) $  and
 $H$ is the mean curvature of $\Sigma_t$.
 Moreover,  $\Sigma_t$ is  star-shaped with respect to  $P$  and
 the rescaled hypersurface  $ \widetilde{\Sigma_t }$, parametrized by $ \widetilde{x}(\cdot , t) =
  e^{- \frac{t}{n-1} } x( \cdot, t)   $, converges to a sphere  centered at $P$ in the $ \mathcal{C}^\infty$ topology
 as $ t \rightarrow \infty$.
\end{theorem}
\begin{proof}[Proof of Theorem \ref{ineq_rn}]
We consider $\lambda=r$ in (\ref{monotone_quantity}). In this case, since $\lambda(0)=0$, we obtain that the quantity
$$Q(t)=|\Sigma_t|^{-\frac{n}{n-1}}\left(\int_{\Sigma_t}r\,d\Sigma-\mbox{\rm Vol}(\Omega_t)\right)$$
is monotone nonincreasing. The next step is to show that
\begin{equation}\label{limit_in_rn}
\lim_{t\to\infty}Q(t)\geq\frac{n-1}{n}\left(\frac{1}{\omega_{n-1}}\right)^{\frac{1}{n-1}}.
\end{equation}
Since the inequality we want to show is scale invariant and, by Theorem \ref{IMCF_in_rn}, the rescaled IMCF converges to a sphere, in order to show (\ref{limit_in_rn}) we just need to show that the inequality holds if $\Sigma$ is a sphere. This follows from (\ref{warp-ineq}) and the fact that the equality holds in the isoperimetric inequality. Thus, $$Q(0)\geq\frac{n-1}{n}\left(\frac{1}{\omega_{n-1}}\right)^{\frac{1}{n-1}},$$ which is just a rewriting of ($\ref{rn}$).

If $\Sigma$ is an origin centered sphere, a straightforward computation shows that the equality holds in (\ref{rn}).

If the equality holds in (\ref{rn}), then $$Q(0)=\frac{n-1}{n}\left(\frac{1}{\omega_{n-1}}\right)^{\frac{1}{n-1}}.$$ Applying (\ref{rn}) to $\Sigma_t$ we find, on one hand, that 
$$Q(t)\geq\frac{n-1}{n}\left(\frac{1}{\omega_{n-1}}\right)^{\frac{1}{n-1}},$$
for all $t$. On the other hand, from the monotonicity of $Q(t)$, we find that $$Q(t) \leq Q(0) = \frac{n-1}{n}\left(\frac{1}{\omega_{n-1}}\right)^{\frac{1}{n-1}},$$ for all $t$. Thus, 
we obtain
$$Q(t)=\frac{n-1}{n}\left(\frac{1}{\omega_{n-1}}\right)^{\frac{1}{n-1}},\quad \forall t.$$
In particular, $Q'(0) = 0$, which, by Proposition \ref{Key_result}, happens if and only if $\Sigma$ is a slice, which in this case means an origin centered sphere.

\end{proof}

\begin{proof}[Proof of Corollary \ref{k=0}.]

Holder's inequality gives
\begin{equation} \label{holder2}
\left(\int_{\Sigma}r\,d\Sigma\right)^2\leq\left(\int_\Sigma r^2\,d\Sigma\right)|\Sigma|,
\end{equation}
with the equality occurring if and only if $r$ is constant, that is, if and only if $\Sigma$ is a origin centered sphere. Combining (\ref{holder2}) and (\ref{rn}) we find
\begin{align*}
\int_\Sigma r^2\,d\Sigma\geq\frac{1}{|\Sigma|}\left(\frac{n-1}{n}\omega_{n-1}\left(\frac{|\Sigma|}{\omega_{n-1}}\right)^{\frac{n}{n-1}}+\mbox{\rm Vol}(\Omega)\right)^2,
\end{align*}
which is just a rewriting of (\ref{inequality corollary}).

If $\Sigma$ is a origin centered sphere, then it is straightforward to verify that the equality holds in (\ref{inequality corollary}).

If the equality holds in (\ref{inequality corollary}), then it also holds in (\ref{holder2}), which implies that $\Sigma$ is a origin centered sphere.

\end{proof}

\begin{proof}[Proof of Theorem \ref{k=1}]
In \cite{kwong-miao}, using the IMCF, Kwong and Miao obtained the following inequality:
\begin{equation}\label{derivation}
    \frac{d}{dt}\left(\int_{\Sigma_t}r^2H\,d\Sigma\right)\leq 2n\mbox{Vol}(\Omega_t)+\frac{n-2}{n-1}\int_{\Sigma_t}r^2H\,d\Sigma,
\end{equation}
where $H=(n-1)\sigma_1$. This inequality is crucial in the proof of (\ref{Kwong-Miao's inequality}).

Consider the quantity
\begin{equation} \label{E}
\mathcal{E}(\Sigma)=
|\Sigma|^{-\frac{n}{n-1}}\int_{\Sigma} r^2H\,d\Sigma.
\end{equation}
The function $E$ defined by 
$$E(t) = \mathcal{E}(\Sigma_t)$$
satisfies
$$E'(t)\leq0, \ \forall t.$$
In fact,
\begin{eqnarray*}
E'(t)&=&\frac{1}{|\Sigma|^{\frac{2n}{n-1}}}\left[
 \frac{d}{dt}\left(\int_{\Sigma_t}r^2H\,d\Sigma\right)|\Sigma|^{\frac{n}{n-1}}-\left(\int_{\Sigma_t}r^2H\,d\Sigma\right)\frac{n}{n-1}|\Sigma|^{\frac{n}{n-1}}
\right]\\
&=&\frac{1}{|\Sigma|^{\frac{n}{n-1}}}\left[
\frac{d}{dt}\left(\int_{\Sigma_t}r^2H\,d\Sigma\right)-\frac{n}{n-1}\int_{\Sigma_t}r^2H\,d\Sigma
\right]\\
&\leq&\frac{1}{|\Sigma|^{\frac{n}{n-1}}}
\left(
2n\mbox{Vol}(\Omega_t)+\frac{n-2}{n-1}\int_{\Sigma_t}r^2H\,d\Sigma-\frac{n}{n-1}\int_{\Sigma_t}r^2H\,d\Sigma
\right)\\
&=&\frac{2}{|\Sigma|^{\frac{n}{n-1}}}
\left(
n\mbox{Vol}(\Omega_t)-\frac{1}{n-1}\int_{\Sigma_t}r^2H\,d\Sigma
\right)\\
&\leq&0,
\end{eqnarray*}
where we have used (\ref{derivation}) to get the first inequality sign and (\ref{Kwong-Miao's inequality}) to get the second one. 
Moreover, $E(t)$ is constant if and only if the equality holds in (\ref{derivation}) and (\ref{Kwong-Miao's inequality}) for all $t$, which occurs if and only if $\Sigma_t$ is an origin centered geodesic sphere for all $t$.

Notice that, on a round sphere, the value of $\mathcal{E}$ is at least
$$
\frac{n-1}{(\omega_{n-1})^{\frac{1}{n-1}}}.
$$
This follows from (\ref{Kwong-Miao's inequality}) and the fact that, on a round sphere, the equality holds in the isoperimetric inequality.

Now, using the scale invariance of (\ref{E}) and that the normalized IMCF converges to a round sphere, we find
$$\lim_{t\to\infty}E(t)\geq\frac{n-1}{(\omega_{n-1})^{\frac{1}{n-1}}}.$$
Since $E(t)$ is monotone nonincreasing, we obtain $E(0)\geq E(t)$ for all $t$. Hence,
$$
E(0) \geq \frac{n-1}{(\omega_{n-1})^{\frac{1}{n-1}}},
$$
which is just a rewriting of (\ref{improved kwong-miao again}).

If $\Sigma$ is a origin centered sphere, it is straightforward to verify that the equality holds in (\ref{improved kwong-miao again}).

Suppose the equality holds in (\ref{improved kwong-miao again}), that is, suppose $E(0) = 0$. On one hand, applying (\ref{improved kwong-miao again}) to $\Sigma_t$, we find
$$
E(t) \geq \frac{n-1}{(\omega_{n-1})^{\frac{1}{n-1}}},
$$
for all $t$. On the other hand, using the monotonicity of $E$, we find
$$
E(t) \leq E(0) = \frac{n-1}{(\omega_{n-1})^{\frac{1}{n-1}}},
$$
for all $t$. Therefore,
$$
E(t) = \frac{n-1}{(\omega_{n-1})^{\frac{1}{n-1}}},
$$
for all $t$. Thus, $E(t)$ is constant and, as explained above, this implies that each $\Sigma_t$ is an origin centered sphere, for each $t$. In particular, $\Sigma = \Sigma_0$ is an origin centered sphere. 

\end{proof}

\subsection{The Sphere as the ambient space 
}

\begin{proof}[Proof of Theorem \ref{ineq_sn}] 

Without loss of generality, we can assume $x(\Sigma)$ is the origin, since this can always be achieved by applying to $\Sigma$ an isometry of $\mathbb{S}^n$. In this case, inequality (\ref{sn}) takes the following form:

\begin{equation}\label{sn_balanced}
\int_\Sigma\sin r\geq\int_\Omega\cos r +\frac{n-1}{n}\omega_{n-1}\left(\frac{|\Sigma|}{\omega_{n-1}}\right)^{\frac{n}{n-1}}.
\end{equation}

Consider $\lambda=\sin r $ in (\ref{monotone_quantity}). As in the Euclidean case, $\lambda(0)=0$. Thus, we have that
$$Q(t)=|\Sigma_t|^{-\frac{n}{n-1}}\left(\int_{\Sigma_t}\sin r\,d\Sigma-\int_{\Omega_t}\cos r\,d\Omega\right)$$
is monotone nonincreasing.

It is proved in \cite{makowski-scheuer} that the IMCF is smooth on an interval $[0,T^*)$, with $\Sigma_t$ converging to an equator $E_\Sigma$, as $t\to T^*$.  The next step is to show that
$$\lim_{t\to T^*}Q(t)=\frac{n-1}{n}\left(\frac{1}{\omega_{n-1}}\right)^{\frac{1}{n-1}}.$$
Indeed, since $\Sigma_t$ converges to an equator, we have
$$\lim_{t\to T^*}|\Sigma_t|=\omega_{n-1},\quad
\lim_{t\to T^*}\int_{\Sigma_t}\sin r\,d\Sigma=\omega_{n-1},\quad \lim_{t\to T^*}\int_{\Omega_t}\cos r\,d\Omega=\frac{\omega_{n-1}}{n}.$$
Thus, 
\begin{align*}
\lim_{t\to T^*}Q(t)&=(\omega_{n-1})^{-\frac{n}{n-1}}\left(\omega_{n-1}-\frac{\omega_{n-1}}{n}\right)=\frac{n-1}{n}\left(\frac{1}{\omega_{n-1}}\right)^{\frac{1}{n-1}}.
\end{align*}
From the monotonicity of $Q(t)$, we have
$$Q(0)\geq Q(t),$$
for all $t\in[0,T^*)$. Appling the limit as $t \to T^\ast$ we obtain
$$Q(0)\geq\frac{n-1}{n}\left(\frac{1}{\omega_{n-1}}\right)^{\frac{1}{n-1}},$$
which is a rewriting of (\ref{sn_balanced}).

If $\Sigma$ is an origin centered geodesic sphere, it is straightforward to verify the equality in (\ref{sn_balanced}).

If the equality holds in (\ref{sn_balanced}), then $Q(0)=\frac{n-1}{n}\left(\frac{1}{\omega_{n-1}}\right)^{\frac{1}{n-1}}$.
Applying (\ref{sn_balanced}) to $\Sigma_t$ we find, on one hand, that $Q(t) \geq \frac{n-1}{n}\left(\frac{1}{\omega_{n-1}}\right)^{\frac{1}{n-1}}$, for all $t \in [0,T^\ast)$. On the other hand, from the monotonicity of $Q(t)$, we find that $Q(t) \leq Q(0) = \frac{n-1}{n}\left(\frac{1}{\omega_{n-1}}\right)^{\frac{1}{n-1}}$, for all $t \in [0,T^\ast)$. Thus, 
we obtain
$$Q(t)=\frac{n-1}{n}\left(\frac{1}{\omega_{n-1}}\right)^{\frac{1}{n-1}},\quad \forall t \in [0,T^\ast).$$
In particular, $Q'(0) = 0$, which, by Proposition \ref{Key_result}, happens if and only if $\Sigma$ is a slice, which in this case means an origin centered geodesic sphere.

\end{proof}


\subsection{The adS-Reissner-Nordstr\"om as the ambient space 
} 
Let $\Sigma$ be a mean convex star-shaped hypersurface in $P$. It was proved in \cite{wang-z} and more recently in \cite{chen-li-zhou} that the solution of the inverse mean curvature flow is smooth and defined on $[0,\infty)$.

The following lemma describes the asymptotic behaviour of several geometric quantities.

\begin{lemma}
Let $g$ be the induced metric on $\Sigma$. The following asymptotic behaviour occurs:
\begin{align}
   \lambda &= O(e^{\frac{1}{n-1}t}), \label{asymptotic lambda} \\
    \sqrt{{\rm det}g} &=\lambda^{n-1}\sqrt{{\rm det}h_S}\left(1+O(e^{-\frac{2}{n-1}t}), \right) \label{asymptotic det g} \\
    |\Sigma| &=\left(\int_{S}\lambda^{n-1}dS\right)\left(1+O\left(e^{-\frac{2}{n-1}t}\right)\right), \label{asymptotic area} \\
    \int_\Sigma\lambda\,d\Sigma &=\left(\int_{S}\lambda^n dS\right)\left(1+O\left(e^{-\frac{2}{n-1}t}\right)\right) \ \ \text{and} \label{asymptotic integral lambda} \\
    |\Sigma|^{\frac{n}{n-1}} &=\left(\int_{S}\lambda^{n-1}dS\right)^{\frac{n}{n-1}}\left(1+O\left(e^{-\frac{2}{n-1}t}\right)\right). \label{asymptotic power of area}
    \end{align}
\end{lemma}

\begin{proof}
Identities (\ref{asymptotic lambda}) and (\ref{asymptotic det g}) are proved in \cite{chen-li-zhou} (Lemma 3.1 and Lemma 4.1, respectively). To get (\ref{asymptotic area}), just integrate (\ref{asymptotic det g}). To get (\ref{asymptotic integral lambda}), multiply (\ref{asymptotic det g}) by $\lambda$ and integrate it.

It remains to show (\ref{asymptotic power of area}). Denote by $A$ the quantity 
$$
\frac{|\Sigma|}{\displaystyle\int_{S}\lambda^{n-1}dS.}
$$

From (\ref{asymptotic area}) we have
$
A = 1 + \alpha,
$
with 
$
\alpha = O(e^{-\frac{2}{n-1}t}).
$

Now, consider the function $f(x) = x^{\frac{n}{n-1}}$. Since $f$ is differentiable, we have
$$
f(1 + \alpha) - f(1) - f'(1)\cdot \alpha = o(\alpha).
$$
Hence,
$$
A^{\frac{n}{n-1}} = 1 + O(e^{-\frac{2}{n-1}t}).
$$
Thus, we get (\ref{asymptotic power of area}).
\end{proof}



\begin{proof}[Proof of Theorem \ref{ineq_ads}] Consider $\lambda$ defined by (\ref{ode}). We have $\lambda(r(s_0))=s_0$ and $\lambda'(r(s_0))=0$. Thus, (\ref{monotone_quantity}) given by
$$
 Q(t)=|\Sigma_t|^{-\frac{n}{n-1}}\left(\int_{\Sigma_t}\lambda\,d\Sigma-\int_{\Omega_t}\lambda'\,d\Omega-\frac{s_0}{n}|\partial P|\right)
$$
is monotone nonincreasing. We will show that
\begin{equation}\label{limit}
\lim_{t\to\infty}Q(t)\geq\frac{n-1}{n}\left(\frac{1}{\vartheta_{n-1}}\right)^{\frac{1}{n-1}}.
\end{equation}
By (\ref{warp-ineq}) we have
$$\int_\Sigma\lambda\,d\Sigma-\int_\Omega\lambda'\,d\Omega-\frac{s_0}{n}|\partial P|\geq\frac{n-1}{n}\int_\Sigma\lambda\,d\Sigma.$$
Thus, to show (\ref{limit}), it is enough to show that
\begin{equation} \label{liminf}
\lim\inf \frac{\displaystyle\int_\Sigma\lambda\,d\Sigma}{|\Sigma_t|^{\frac{n}{n-1}}}\geq\left(\frac{1}{\vartheta_{n-1}}\right)^{\frac{1}{n-1}}.
\end{equation}
From (\ref{asymptotic integral lambda}) and (\ref{asymptotic power of area}) we find
\begin{align*}
\lim\inf \frac{\displaystyle\int_\Sigma\lambda\,d\Sigma}{|\Sigma_t|^{\frac{n}{n-1}}}=\lim\inf\frac{\int_{S}\lambda^{n}dS}{\left(\int_{S}\lambda^{n-1}dS\right)^{\frac{n}{n-1}}}.
\end{align*}
But Holder's inequality gives
\begin{equation*}
(\vartheta_{n-1})^{\frac{1}{n-1}}\int_{S} \lambda^n dS \geq \left( \int_{S} \lambda^{n-1} dS \right)^{\frac{n}{n-1}},
\end{equation*}
which implies (\ref{liminf}).
This proves inequality (\ref{ads}).

If $\Sigma = \lbrace s \rbrace \times S$, for some $s \in [s_0,\infty)$, a straightforward computation shows that the equality holds in (\ref{ads}).

If the equality holds in (\ref{ads}), then $$Q(0)=\frac{n-1}{n}\left(\frac{1}{\vartheta_{n-1}}\right)^{\frac{1}{n-1}}.$$ Applying (\ref{ads}) to $\Sigma_t$, we find, on one hand, that $$Q(t) \geq \frac{n-1}{n}\left(\frac{1}{\vartheta_{n-1}}\right)^{\frac{1}{n-1}},$$ for all $t$. On the other hand, from the monotonicity of $Q(t)$, we find that $$Q(t) \leq Q(0) = \frac{n-1}{n}\left(\frac{1}{\vartheta_{n-1}}\right)^{\frac{1}{n-1}},$$ for all $t$. Thus, we obtain
$$
Q(t) = \frac{n-1}{n}\left(\frac{1}{\vartheta_{n-1}}\right)^{\frac{1}{n-1}}, \ \forall t.
$$
In particular, $Q'(0) = 0$, which, by Proposition \ref{Key_result}, happens if and only if $\Sigma$ is a slice.


\end{proof}
\section{A surface with small polar moment of inertia} \label{section_counterexample}

The purpose of this section is to construct a star-shaped and strictly mean convex surface $\Sigma$ in $\mathbb{R}^3$ for which
$$
|\Sigma|^{-2}\int_\Sigma r^2\,d\Sigma<\frac{1}{\omega_2},
$$
given a conterexample to (\ref{improved kwong}) when $n=3$.

Our construction is inspired on examples of surfaces of constant width due to Fillmore \cite{Fillmore}. Our example is obtained by rotating the curve

$$
\left\{
\begin{array}{rcl}
x&=&(\cos3t+9)\sin t-3\sin3t\cos t\\
y&=&(\cos3t+9)\cos t+3\sin3t\sin t
\end{array}\right.,\quad0\leq t\leq2\pi.
$$
around the $y$-axis.

\begin{figure}
\begin{minipage}{0.4\textwidth}
\includegraphics[scale=.4]{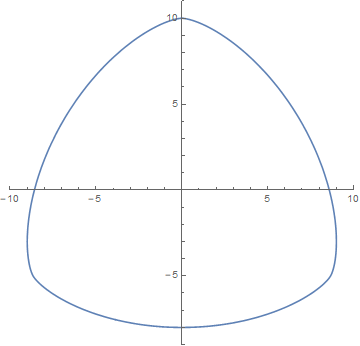}
\end{minipage}%
\begin{minipage}{0.6\textwidth}
\includegraphics[scale=.4]{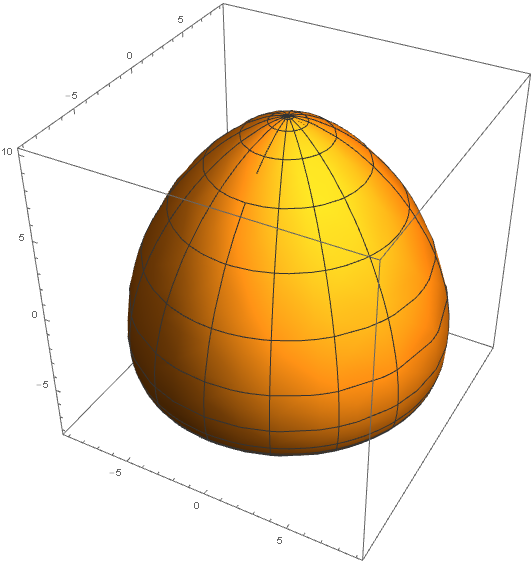}
\end{minipage}
\caption{Constant width curve and surface generated by its rotation around the $y$-axis}
\end{figure}

Such surface is analytic and described by
$$
\left\{
\begin{array}{rcl}
x&=&[(\cos3 t+9)\sin t-3\sin3 t\cos t]\cos s\\
y&=&[(\cos3 t+9)\sin t-3\sin3t\cos t]\sin s\\
z&=&(\cos3t+9)\cos t+3\sin3t\sin t.
\end{array}
\right.,\ 0\leq t\leq\pi\ \mbox{and}\ 0\leq s\leq2\pi.
$$

After some computations we obtain
$$
|\Sigma|^2=\frac{122855056\pi^2}{1225}$$
and
$$\int_\Sigma r^2\,d\Sigma=\frac{124744936\pi}{5005}.$$
Thus,
$$|\Sigma|^{-2}\int_\Sigma r^2\,d\Sigma=\frac{545759095}{2196034126 \pi }<\frac{1}{4\pi}=\frac{1}{\omega_2}.$$ 
Denoting by $\kappa_1$ and $\kappa_2$ the principal curvatures of the previous surface, we have
$$\kappa_1=\frac{1}{\sqrt{(9-8 \cos (3 t))^2}}\quad{\rm and}\quad\kappa_2=\frac{8 \cos (3 t)-9}{\left(8 \cos ^3(t)-9\right) \sqrt{(9-8 \cos (3 t))^2}}.$$
It is not hard to see that
$$\frac{1}{17}\leq \kappa_1,\kappa_2\leq 1.$$
Thus, $\Sigma$ is strictly convex.




\end{document}